\documentclass[final,notitlepage,12pt,reqno]{amsart}
\usepackage{amsmath,colonequals,url,extarrows}
\DeclareMathOperator{\D}{d}

\DeclareMathOperator{\R}{Re}
\usepackage{amsthm}

\theoremstyle{plain}
\newtheorem{theorem}{Theorem}[section]

\newtheorem{lemma}[theorem]{Lemma}

\newtheorem{example}[theorem]{Example}

\newenvironment{remark}[1][Remark]{\begin{trivlist}
\item[\hskip \labelsep {\bfseries #1}]}{\end{trivlist}}

\theoremstyle{definition}

\numberwithin{equation}{section}
\usepackage{color}

\def\eor{\hfill$ \square$}
\usepackage{Baskervaldx}
\usepackage[baskervaldx]{newtxmath}
\usepackage[cal=cm,scr=rsfs,frak=euler]{mathalfa}

\usepackage[french,german,romanian,russian,english]{babel}
\begin{document}

\title[Multiple elliptic integrals and differential equations]{Multiple elliptic integrals\\ and differential equations}
\author{John M. Campbell}
\address[J. M. Campbell]{Department of Mathematics and Statistics, Dalhousie University, Halifax, NS, B3H 4R2, Canada}\email{jmaxwellcampbell@gmail.com}
\author{M. Lawrence Glasser}
\address[M. L. Glasser]{Department of Physics,
Clarkson University, Potsdam NY 13699, USA}
\email{laryg@clarkson.edu}
 \author{Yajun Zhou
}
\address[Y. Zhou]{Program in Applied and Computational Mathematics (PACM), Princeton University, Princeton, NJ 08544, USA} \email{yajunz@math.princeton.edu}\curraddr{\textrm{} \textsc{Academy of Advanced Interdisciplinary Studies (AAIS), Peking University, Beijing 100871, P. R. China}}\email{yajun.zhou.1982@pku.edu.cn}
\date{\today}\thanks{\indent * The research of J.\ M.\ C.\ was supported via a Killam Postdoctoral Fellowship from the Killam Trusts. The research of Y.\ Z.\ was supported in part by the Applied Mathematics Program within the Department of Energy
(DOE) Office of Advanced Scientific Computing Research (ASCR) as part of the Collaboratory on
Mathematics for Mesoscopic Modeling of Materials (CM4)}
\maketitle

\begin{abstract}
 We {introduce and prove}
 evaluations {for families of} multiple elliptic integrals by solving special types of ordinary and partial differential 
 equations. As an application, we obtain new expressions of Ramanujan-type series of level 4 and associated singular 
 values for the complete elliptic integral $\text{{\bf K}}$ with integrals involving  $\text{{\bf K}}$. 
\end{abstract}

\section{Introduction}
 \textit{{M}ultiple elliptic integrals} {include expressions of the form} 
\begin{align}
 \int_\alpha^\beta F(u) \mathbf K(G(u))\D u,\label{eq:MultEllInt_defn} 
 \end{align} for algebraic functions   $F(u)$  and $G(u)$,  and 
\begin{align}
 \mathbf K(k)\colonequals\int_{0}^{\pi/2}\frac{\D \theta}{\sqrt{1 - 
 k^2\sin^2\theta}}\label{eq:K_defn} 
\end{align}is the \textit{complete elliptic integral} of the first kind with \textit{modulus} $k$. {M}ultiple elliptic integrals as in 
 {\eqref{eq:MultEllInt_defn},} 
 together with more sophisticated variations whose integrands involve higher powers of $ \mathbf K${,}
 captivated the interests of all the authors of this paper {over the decades}
 \cite{Bailey2008,Campbell2023CCG,Glasser1976,Zhou2013Pnu}{, with} such integrals {arising} from 
 various applications to physics, including the vibrational properties of lattices \cite{Glasser1976} and the interactions of elementary particles \cite{Bailey2008}. 
 In this note, we deduce evaluations 
 for two families of multiple elliptic integrals  of the form  shown in  \eqref{eq:MultEllInt_defn}, via analytic solutions to {ordinary and partial}
 differential equations.

Our main mission in \S\ref{sec:IntODE} is to exploit a third-order ordinary differential equation (ODE) of Picard--Fuchs type in the construction of the following integral formula with a tunable parameter $ a\in[0,1]$:
\begin{align}
 \int_0^1\frac{\mathbf K(2\sqrt{x(1-x)})\D x}{\sqrt{1-2(2 x - 
 1) a + a^{2}}} = \left[\mathbf K\left( \sqrt{\frac{1-\sqrt{1+a^{2}}}{2}} \right)\right]^2.\label{eq:intODE} 
\end{align}
 Here, on the right-hand side of \eqref{eq:intODE}, it does not matter which univalent branch we choose for the square root of negative numbers, since the complete elliptic integral $ \mathbf K(k)$ defined in \eqref{eq:K_defn} is effectively a function of $ k^2$. As we differentiate \eqref{eq:intODE} before specializing to some critical values of $a$,  known  Ramanujan-type series will enable us to produce 
 {new} multiple elliptic integral {evaluations}, such as 
\begin{align}
 \int_0^1\frac{\mathbf K(2\sqrt{x(1-x)})(4x+3\sqrt{2}-2)\D x}{(4\sqrt{2}+9-8\sqrt{2}x)^{3/2}} = 
 \frac{\pi}{4\sqrt{2}}. \label{eq:intODE'}
\end{align} 
 {Moreover, special values for the elliptic lambda function corresponding to Ramanujan-type 
 series for $\frac{1}{\pi}$ may be applied, together 
 with elliptic integral singular values, to obtain 
 new valuations for multiple elliptic integrals, including 
\begin{align}\label{singularmotivating}
 \int_{0}^{1} \frac{\text{{\bf K}}( 2 \sqrt{x(1-x)} ) \D x }{\sqrt{\frac{9}{8} + \frac{1 - 2x}{\sqrt{2}}} } 
 = \frac{\big[\Gamma \big(\frac{1}{4}\big)\big]^4}{16 \sqrt{2} \pi }, 
\end{align}}where $ \Gamma(s)\colonequals \int_0^\infty t^{s-1}e^{-t}\D t$ is Euler's gamma function for $ \R s>0$.

In \S\ref{sec:IntPDE}, we take advantage of Laplace's equation, 
 which is a second-order partial differential equation (PDE), and build an integral formula with two tunable parameters $ b,c\geq0$:
\begin{align}
\begin{split}
 & \int_0^{\pi/2}\mathbf K\left( \sqrt{\frac{4c\tan\theta}{b^{2}+(c + 
 \tan\theta)^2}} \right)\frac{\sin\theta\D \theta}{\sqrt{b^{2}+(c+\tan\theta)^2}} \\
={} & \frac{\pi}{2\sqrt{(b+1)^{2}+c^2}}.
\end{split}\label{eq:intPDE}
\end{align}
 A special case of the integral formula above can be rearranged into
\begin{align} 
 \int_0^1\frac{\mathbf K(2\sqrt{x(1-x)})x(1-x)\D x}{[1-2x(1-x)]^{3/2}}=\frac{\pi}{2\sqrt{2}},\label{eq:intPDE'} 
\end{align}
 which hearkens back to \eqref{eq:intODE'}.

\section{Multiple elliptic integrals via a third-order ODE\label{sec:IntODE}}
 In what follows, we set 
 \textit{Pochhammer symbols} 
 as $(x)_0\colonequals 1 $ and  $ (x)_n\colonequals x(x+1)\cdots(x+n-1)$ for positive integers $ n$. 
 
We begin by verifying \eqref{eq:intODE} at a special point where $a=0$. 

\begin{lemma}\label{lm:a=0}
 We have an integral identity
\begin{align} 
 \int_0^1\mathbf K(2\sqrt{x(1-x)})\D x=\frac{\pi^{2}}{4}.\label{eq:quarter_pi_sqr} 
\end{align}
\end{lemma}

\begin{proof}
 We can reformulate \eqref{eq:K_defn} into the following Maclaurin expansion:\begin{align}
\mathbf K(k)=\frac{\pi}{2}\left[ 1+\sum_{n=1}^\infty\frac{\big(\frac12\big)_n^2}{(1)_n^2}k^{2n}\right],\quad |k^2|<1. 
\end{align}
 {We obtain} $ \frac{\pi}{2}\left[ 1+\sum_{n=1}^\infty\frac{\left(\frac12\right)_n^2}{(1)_n^2}\frac{4^n (n!)^2}{(2 n+1)!}\right]$
 via termwise integration applied to \eqref{eq:quarter_pi_sqr}, 
 which {may be evaluated as} $\frac{\pi^2}{4} $ 
 {according to the Maclaurin series for} 
 $\frac{\arcsin\sqrt{x}}{\sqrt{x}}$. 
\end{proof}
 
 With {Lemma \ref{lm:a=0},} we can now establish \eqref{eq:intODE} in its entirety. 

\begin{proof}[Proof of \eqref{eq:intODE} by ODE]Typing \begin{quote}\texttt{Annihilator[Integrate[EllipticK[4*(1 - x)*x]/Sqrt[1 - 2*(2*x-1)*a + a*a], \{x, 0, 1\}], Der[a]]}\end{quote} in Koutschan's \texttt{HolonomicFunctions} package v1.7.3 \cite{Koutschan2013}, one sees that the left-hand side of \eqref{eq:intODE} is annihilated by a third-order differential operator\begin{align}
\widehat L_a\colonequals a^{2}\big(1+a^{2}\big) \frac{\D^3}{\D a^3}+3 a\big(1+2a^2\big) \frac{\D^2}{\D a^2}+\big(1+7 a^2\big) \frac{\D}{\D a}+a. 
\end{align}The homogeneous Picard--Fuchs equation $ \widehat L_a f(a)=0,a\in(0,1)$ is solved by \begin{align}
\begin{split}
&f(a)\\={}&c_{1}\left[ \mathbf K\left( \sqrt{\frac{1-\sqrt{1+a^{2}}}{2}} \right) \right]^2+c_2\mathbf K\left( \sqrt{\frac{1-\sqrt{1+a^{2}}}{2}} \right)\mathbf K\left( \sqrt{\frac{1+\sqrt{1+a^{2}}}{2}} \right)\\{}&+c_3 \left[ \mathbf K\left( \sqrt{\frac{1+\sqrt{1+a^{2}}}{2}} \right) \right]^2,
\end{split}
\end{align}where $c_1$, $c_2$, and $c_3$ are constants.  Here, we must have $ c_2=c_3=0$, because\begin{align}
 \mathbf K\left( \sqrt{\frac{1-\sqrt{1+a^{2}}}{2}} \right) ={}&\frac{\pi}{2}+O(a^2),\\\mathbf K\left( \sqrt{\frac{1+\sqrt{1+a^{2}}}{2}} \right)={}&-\frac{\pi i}{2} +\log \frac{8}{a}+O(a^2\log a),
\end{align}as $ a\to0^+$, while the left-hand side of \eqref{eq:intODE} remains bounded as $a$ tends to zero. Thus, we have determined the right-hand side of \eqref{eq:intODE} up to a multiplicative constant $c_1$. Lemma \ref{lm:a=0} reveals that $c_1=1$. 
\end{proof}
\begin{remark}For $ a\in[1,\infty)$, one can manipulate the integrand on the left-hand side of \eqref{eq:intODE} into a form related to $ a\in(0,1]$, so the right-hand side of \eqref{eq:intODE} should be modified into\begin{align}
\frac{1}{a}\left[ \mathbf K\left( \sqrt{\frac{1-\sqrt{1+a^{-2}}}{2}} \right)\right]^2
\end{align} when $a\geq1$. In particular, we note that the integral in \eqref{eq:intODE} does not admit a smooth continuation across the critical point $a=1$. 
\eor\end{remark}
 
  Defining  \textit{Legendre polynomials} \cite[\S10]{Rainville1960} through  
\begin{align}
P_{n}(x)\colonequals \frac{1}{2^n}\sum_{k=0}^n\binom{n}{k}^2(x - 
 1)^{n-k}(x+1)^k,\label{eq:Pn(x)_defn}
\end{align}
 we will construct {an alternative proof of \eqref{eq:intODE} via the expansion}
\begin{align}
 \sum_{n=0}^\infty(-1)^n\frac{\big(\frac12\big)^3_n}{(1)_n^3}(4n+1)P_{2n}(2x-1)=\frac{4\mathbf K(2\sqrt{x(1-x)})}{\pi^2} \label{eq:Baranov_expn}
\end{align}
{considered, for example by}
 Baranov \cite{Baranov2006}, 
 {together with} 
 the generating function for Legendre polynomials 
\begin{align}
 \frac{1}{\sqrt{1-2(2x-1)a+a^2}}=\sum_{n=0}^\infty P_n(2x-1)a^n,\quad |a|<1. \label{eq:Pn_gfun}
\end{align}
 Concretely speaking, applying the orthogonality
relations\begin{align}
\int_0^1 P_n(2x-1)P_m(2x-1)\D x=\begin{cases}0 & \text{if }n\neq m, \\
\frac{1}{2n+1} & \text{otherwise} \\
\end{cases}
\label{eq:Pn_ortho}
\end{align} to \eqref{eq:Baranov_expn} and \eqref{eq:Pn_gfun}, we may identify the left-hand side of \eqref{eq:intODE} with \begin{align}
\frac{\pi^{2}}{4}\left[1+\sum_{n=1}^\infty(-1)^n\frac{\big(\frac12\big)^3_n}{(1)_n^3}a^{2n}\right].\label{eq:Pn_sum}
\end{align} Such a sum is equal to the right-hand side of \eqref{eq:intODE}, according to the \textit{Clausen product identity} [cf.\ \eqref{Clausen} below]. 

A remarkable aspect of 
 the elliptic integral identity in \eqref{eq:intODE} is given by an evaluation that we obtain, as below, 
 through the use of the modular relation 
\begin{equation}\label{complexmodular}
 \text{{\bf K}}\left( i \frac{k}{k'} \right) = k' \text{{\bf K}}(k).   
\end{equation}
 With regard to the definition in \eqref{eq:K_defn}, 
  we let $k' := \sqrt{1-k^2}$ 
 denote the \emph{complementary modulus}, 
 and we write $\text{{\bf K}}'(k) := \text{{\bf K}}(k')$. 
 The \emph{elliptic lambda function} $\lambda^{\ast}$ may be defined so that 
 $\lambda^{\ast}(r)$ for $r > 0$ is the unique value such that 
 $ \frac{ \text{{\bf K}}'\left( \lambda^{\ast}(r) \right) }{ 
 \text{{\bf K}}\left( \lambda^{\ast}(r) \right) } = \sqrt{r}$, 
 referring to the classic \emph{Pi and the AGM} text for details \cite[\S3.2]{BorweinBorwein1987}. 
 For rational values of $r > 0$, the expression $\text{{\bf K}}\left( \lambda^{\ast}(r) \right)$ 
 admits an evaluation as an algebraic multiple of a combination of values of Euler's gamma function with 
 rational arguments, 
 again with reference to the \emph{Pi and the AGM} text \cite[\S9.2]{BorweinBorwein1987}. 
 By applying these singular values together with the modular equation in \eqref{complexmodular} 
 and the elliptic integral identity in \eqref{eq:intODE}, 
 we obtain evaluations as in the motivating example highlighted in \eqref{singularmotivating}. 

\begin{example}
 Using the valuation 
 $ \lambda^{\ast}(4) = 3 - 2 \sqrt{2} $ 
 together with the relations in \eqref{eq:intODE} and \eqref{complexmodular} 
 and the singular value for $\text{{\bf K}}\big( \lambda^{\ast}(r) \big)$, 
 we obtain the motivating example in \eqref{singularmotivating}. 
\end{example}

\begin{example}
 Using the valuation 
 $\lambda^{\ast}(3) = \frac{1}{4} \sqrt{2}(\sqrt{3} - 1)$
 but without the modular relation 
 for $\text{{\bf K}}$ with a complex argument, 
 the multiple elliptic integral identity in \eqref{eq:intODE} can be applied to obtain that 
\begin{equation}\label{notcomplex3}
 \int_{0}^{1} \frac{\text{{\bf K}}( 2 \sqrt{x(1-x)} ) \D x 
 }{\sqrt{3 + 4 i (1 - 2 x)}} 
 = \frac{\sqrt{3} \big[\Gamma\big( \frac{1}{3} \big)\big]^{6} }{2^{\frac{17}{3}} \pi^2}. 
\end{equation}
\end{example}

\begin{example}
 From the valuation $\lambda^{\ast}(7) = \frac{1}{8} \sqrt{2} (3 - \sqrt{7})$, 
 by analogy with the derivation of the multiple elliptic integral in \eqref{notcomplex3}, 
 we can obtain that 
 \begin{align}
 \int_{0}^{1} \frac{\text{{\bf K}}( 2 \sqrt{x (1-x)} ) \D x 
 }{ \sqrt{63 + 16 i (1-2x)} } 
 = \frac{\big[\Gamma\big(\frac{1}{7}\big) \Gamma\big(\frac{2}{7}\big)
 \Gamma \big(\frac{4}{7}\big)\big]^2}{128 \sqrt{7} \pi ^2}. 
\end{align}
\end{example}

 Relative to our derivations of the above Examples, 
 we can apply a similar approach 
 to express Ramanujan-type series of level $4$ 
 via the elliptic integral identity in 
 \eqref{eq:intODE}. This can be accomplished by rewriting the right-hand side of 
 \eqref{eq:intODE} according to 
\begin{align}\label{Clausen}
 1+\sum_{n=1}^{\infty} \frac{\big( \frac{1}{2} \big)_{n}^{3} }{\left( 1 \right)_{n}^{3} } \big( \!-\!a^2 \big)^{n} 
 = \frac{4}{\pi^2} \left[\mathbf K\left( \sqrt{\frac{1 - \sqrt{1 + a^2}}{2}} \right)\right]^2. 
\end{align}
 The hypergeometric identity \eqref{Clausen} derived from Clausen's hypergeometric product identity 
 and is key to the classical derivations of Ramanujan's series for $\frac{1}{\pi}$ \cite{Ramanujan1914}. 
 This and the term-by-term derivates of \eqref{Clausen}
 allow us to apply \eqref{eq:intODE} to rewrite Ramanujan's series of level 4 
 as multiple elliptic integrals. 

\begin{example}
 Guillera \cite{Guillera2006} and Baruah--Berndt \cite{BaruahBerndt2010} evaluated \begin{align}
1+\sum_{n=1}^\infty\frac{\big(\frac12\big)^3_n}{(1)_n^3}(6n+1)\left( -\frac{1}{8} \right)^{n}=\frac{2\sqrt{2}}{\pi},
\end{align}
which is a Ramanujan-type  series.  In view of this, we can take a $ \mathbb Q$-linear combination of 
 \eqref{Clausen} and its derivative with respect to $a$, before specializing to $ a=\frac{1}{\sqrt{8}}$ and reaching \eqref{eq:intODE'}. 
\end{example}
 
\begin{example}
 By mimicking our derivation of the motivating result in \eqref{eq:intODE'} 
 with the use of a Ramanujan-type series 
\begin{align} 
 \frac{4\sqrt{2}}{\pi} = \sum_{n=0}^{\infty} \left( -\frac{26 - 15 \sqrt{3}}{16} \right)^{n} 
 \big[(30 - 6 \sqrt{3}) n + 7 - 3 \sqrt{3} \big] \frac{\big( \frac{1}{2} \big)_{n}^{3} }{\left( 1 \right)_{n}^{3}}, 
\end{align}
 introduced by Baruah and Berndt  \cite{BaruahBerndt2010}, 
 we can show that this formula is equivalent to 
\begin{align}
\begin{split}
 & -\frac{\pi }{8 \sqrt{2}} \\={}
 & \int_{0}^{1} \frac{ \text{{\bf K}}( 2 \sqrt{(1-x)x} ) \big[ 24-18 \sqrt{3}+\sqrt{2} \big(6 \sqrt{3}-11\big) (2 
 x-1) \big] \D x }{\big[42-15 \sqrt{3}-4 \sqrt{2} \big(3 \sqrt{3}-5\big) (2 x-1)\big]^{3/2}}. 
\end{split}\end{align}
\end{example}

 \section{Multiple elliptic integrals via a second-order PDE\label{sec:IntPDE}} 
  The Laplace operator\begin{align} 
 \nabla^2\colonequals \frac{\partial^2}{\partial x^2}+\frac{\partial^2}{\partial y^2}+\frac{\partial^2}{\partial z^2}
\end{align}becomes \begin{align}
\nabla^2=\frac{\partial^{2}}{\partial\rho^{2}}+\frac{1}{\rho}\frac{\partial}{\partial\rho}+\frac{\partial^2}{\partial z^2}
\end{align} when we investigate an axially symmetric problem using cylindrical coordinates $ \rho=\sqrt{x^2+y^2}$ and $z$. One can easily verify that \begin{align}
\nabla^2\frac{1}{\sqrt{\rho^2+(z-z_0)^2}}=0
\end{align}holds for any constant $z_0 $, so long as $ \rho^2+(z-z_0)^2\neq0$.\begin{proof}[Proof of \eqref{eq:intPDE} by PDE]Using \texttt{Mathematica} v14.0, one can check that the integrand of \eqref{eq:intPDE}, namely \begin{align}
\mathbf K\left( \sqrt{\frac{4c\tan\theta}{b^{2}+(c+\tan\theta)^2}} \right)\frac{\sin\theta}{\sqrt{b^{2}+(c+\tan\theta)^2}}
\end{align} is annihilated by a partial differential operator\begin{align}
\frac{\partial^{2}}{\partial b^{2}}+\frac{\partial^{2}}{\partial c^{2}}+\frac{1}{c}\frac{\partial}{\partial c}.
\end{align}Therefore, the difference between the two sides of \eqref{eq:intPDE} satisfies the Laplace equation in cylindrical coordinates $c$ and $b$. Such a difference vanishes in three separate scenarios:\begin{itemize}
\item 
When $c=0$, the integral in \eqref{eq:intPDE} is elementary; \item When $ b=0$, one can compute\begin{align}
\begin{split}&
\int_0^{\pi/2}\mathbf K\left( \sqrt{\frac{4c\tan\theta}{(c+\tan\theta)^2}} \right)\frac{\sin\theta\D \theta}{c+\tan\theta}\\\xlongequal{\theta=\arctan(c x)}{}&\int_0^\infty\mathbf K\left( \frac{2\sqrt{x}}{1+x} \right)\frac{c x\D x}{(1+x)(1+c^{2}x^{2})^{3/2}}\\\xlongequal{\text{\cite[item 164.02]{ByrdFriedman}}}{}&\int_0^\infty \frac{\R[\mathbf K(x)]c x\D x}{(1+c^{2}x^{2})^{3/2}}\\\xlongequal{\text{\cite[(21)]{Glasser1976}}}{}&\frac{\pi}{2\sqrt{1+c^2}}; 
\end{split}
\end{align}\item When $ b+c=|b|+|c|\to\infty$, the dominated convergence theorem is applicable to the integral in \eqref{eq:intPDE}, so that both sides of \eqref{eq:intPDE} agree up to $ o(1/\sqrt{b^2+c^2})$. 
\end{itemize}We may thus conclude our proof of \eqref{eq:intPDE} by appealing to the uniqueness principle for solutions to the Laplace equation.
\end{proof}

As we set $ \theta=\arctan\frac{x}{1-x}$, $ b=0$, and $c=1$ in \eqref{eq:intPDE}, we get \eqref{eq:intPDE'}.


\end{document}